\documentclass[10pt,leqno,twoside]{amsart}
\usepackage{}
\usepackage{amsfonts}
\numberwithin{equation}{section}
\usepackage{latexsym}
\usepackage{amsmath}
\usepackage{amssymb}
\usepackage{mathrsfs}
\usepackage{graphicx}
\usepackage{ifthen}
\usepackage[T1]{fontenc}
\usepackage[latin1]{inputenc}
\usepackage{color}
\newtheorem{Satz}{Theorem}[section]
\newtheorem{lemma}[Satz]{Lemma}
\newtheorem{proposition}[Satz]{Proposition}

\newtheorem{theorem}[Satz]{Theorem}

\newtheorem{remarks}[Satz]{Remarks}

\newtheorem{Def}[Satz]{Definition}

\newcommand{\cS}{{\mathcal S}}

\begin{document}

\newcommand{\C}{\mathbb{C}}
\newcommand{\N}{\mathbb{N}}
\newcommand{\R}{\mathbb{R}}
\newcommand{\Z}{\mathbb{Z}}
\newcommand{\hp}{\mathbb{P}}

\renewcommand{\AA}{\mathcal{A}}
\newcommand{\cB}{\mathcal{B}}
\newcommand{\DD}{\mathcal{D}}
\newcommand{\FF}{\mathcal{F}}
\newcommand{\HH}{\mathcal{H}}
\newcommand{\II}{\mathcal{I}}
\newcommand{\KK}{\mathcal{K}}
\newcommand{\LL}{\mathcal{L}}
\newcommand{\MM}{\mathcal{M}}
\newcommand{\PP}{\mathcal{P}}
\newcommand{\QQ}{\mathcal{Q}}
\newcommand{\RR}{\mathcal{R}}
\newcommand{\TT}{\mathcal{T}}
\newcommand{\cU}{\mathcal{U}}
\newcommand{\cG}{\mathcal{G}}
\newcommand{\hitr}{\HH\TT}
\newcommand{\cA}{{\mathcal A}}
\newcommand{\la}{\lambda}
\newcommand{\ve}{\varepsilon}
\newcommand{\vmo}{VMO($\R^n$)}
\newcommand{\vmom}{\mathrm{VMO}(\R^n)}
\newcommand{\bmo}{BMO($\R^n$)}
\newcommand{\bmom}{\mathrm{BMO}(\R^n)}
\newcommand{\vmoN}{VMO($\R^n; \C^{N \times N}$)}
\newcommand{\vmomN}{\mathrm{VMO}(\R^n; \C^{N \times N})}
\newcommand{\bmoN}{BMO($\R^n; \C^{N \times N}$)}
\newcommand{\bmomN}{\mathrm{BMO}(\R^n; \C^{N \times N})}
\newcommand{\Liloc}{L^1_{\rm loc}}
\newcommand{\Lpo}{L^p_\omega(\R^n)}
\newcommand{\Wmpo}{W^{m,p}_\omega(\R^n)}
\newcommand{\mint}[1]{{\int \!\!\!\!\!\! -}_{\!\!#1}\;}
\newcommand{\cz}{Calder\'on-Zygmund}
\newcommand{\vn}{vspace{.2cm}\noindent}
\newcommand{\dd}{{\rm d}}
\newcommand {\<}{\left\langle}
\renewcommand {\>}{\right\rangle}
\renewcommand {\=}[1]{\stackrel{\text{#1}}{=}}
\renewcommand {\Im}{{\rm Im}}
\renewcommand {\Re}{{\rm Re}}
\newcommand{\HR}{Hil\-bert\-raum}
\newcommand{\Hik}{$H^\infty$-calculus}
\newcommand{\K}{\ensuremath\mathbb{K}}
\newcommand{\vr}{\ensuremath\varrho}
\newcommand{\eps}{\ensuremath\varepsilon}
\newcommand{\id}{\ensuremath\mathrm{Id}}
\newcommand{\lap}{\ensuremath\Delta}
\newcommand{\dv}{\mbox{div}\,}
\newcommand{\grad}{\mbox{grad }}
\newcommand{\rot}{\mbox{rot} }
\newcommand{\curl}{\mbox{curl}\,}
\renewcommand{\L}{\mathscr L}
\newcommand{\supp}{{\mathrm {supp}\ }}
\newcommand{\We}{\mbox{We }}
\newcommand{\Rey}{\mbox{Re }}
\renewcommand{\d}{{\ \mathrm {d}}}
\newcommand{\bog}{Bogovski\u{\i}}
\newcommand{\dist}{\ensuremath\mathrm{dist}}
\newcommand{\ran}{\ensuremath\mathrm{ran}}
\newcommand{\diam}{\ensuremath\mathrm{diam}}
\newcommand{\hr}{{\R^n_+}}
\newcommand{\lpb}{{L^p(B(x_0,r))^N}}
\newcommand{\lpbl}{{L^p(B(x_0,|\lambda|^{-1/m}))^N}}
\newcommand{\lpbs}{{L^p(B(x_0,(s+1)r))^N}}
\newcommand{\lp}{{L^p(\R^n)^N}}
\newcommand{\linfb}{{L^\infty(B(x_0,r))^N}}
\newcommand{\linfbl}{{L^\infty(B(x_0,|\lambda|^{-1/m}))^N}}
\newcommand{\lpbsl}{{L^p(B(x_0,(s+1)|\lambda|^{-1/m}))^N}}
\newcommand{\loc}{\mathrm{loc}}
\newcommand{\nn}{\nonumber}

\newcommand{\ie}{\mathcal I_\varepsilon}
\def\vn{\vspace{.2cm}\noindent}
\def\cF{{\mathcal F}}
\def\mF{{\mathfrak F}}
\renewcommand{\theenumi}{\alph{enumi}}
\renewcommand{\labelenumi}{\theenumi)}

\title[Global Existence Results for the Navier-Stokes Equations in the rotational framework]
{Global existence results for the Navier-Stokes equations in the rotational framework}

\author{Daoyuang Fang, Bin Han and Matthias Hieber}

\address{Department of Mathematics, Zhejiang University,
Hangzhou 310027, China}

\address{Department of Mathematics, Zhejiang University,
Hangzhou 310027, China}

\address{Technische Universit\"at Darmstadt, Fachbereich Mathematik,
Schlossgartenstr. 7, D-64289 Darmstadt, Germany and \hfill\break
Center of Smart Interfaces, Petersenstr. 32, D-64287 Darmstadt}

\email{dyf@zju.edu.cn}
\email{hanbinxy@163.com}
\email{hieber@mathematik.tu-darmstadt.de}



\begin{abstract}
Consider the equations of Navier-Stokes in $\R^3$ in the rotational setting, i.e. with
Coriolis force. It is shown that this set of equations admits a unique, global mild
solution provided the initial data is small with respect to the norm the Fourier-Besov space
$\dot{FB}_{p,r}^{2-3/p}(\R^3)$, where $p \in (1,\infty]$ and $r \in [1,\infty]$.

In the two-dimensional setting, a unique, global mild solution to this set of equations exists
for {\em non-small} initial data $u_0 \in L^p_\sigma(\R^2)$ for $p \in [2,\infty)$.
\end{abstract}

\maketitle

\vn
\section{Introduction and Main Results}
Consider the flow of an incompressible, viscous fluid in $\R^3$ in the rotational framework
which is described by the following set of equations
\begin{equation}\label{e1.1}
\left\{
 \begin{array}{rlll}
 \partial_tu+u\cdot\nabla u-\Delta u+\Omega e_3\times u+\nabla \pi&=&0,
\ \ &\mathrm{in}\ \R^3\times (0,\infty),\\
  \hbox{div}u&=&0, \ \ &\mathrm{in}\ \R^3\times (0,\infty),\\
  u(0)&=&u_0,  \ \ &\mathrm{in}\ \R^3.
   \end{array}
  \right.
\end{equation}
Here, $u$ and $\pi$ represent the velocity and pressure of the fluid, respectively, and
$\Omega\in\R$ denotes the speed of rotation around the unit vector $e_3=(0,0,1)$ in
$x_3$-direction. If $\Omega=0$, the system reduces to the classical Navier-Stokes system.

This set of equations recently gained quite some attention due to its importance in applications
to geophysical flows. In particular, large scale atmospheric and oceanic flows are dominated
by rotational effects, see e.g. \cite{Maj03} or \cite{CDGG06}.

If $\Omega=0$, the classical Navier-Stokes equations have been considered by many authors
in various scaling invariant spaces, in particular in
$$
\dot{H}^{\frac{1}{2}}(\R^3)\hookrightarrow L^3(\R^3)\hookrightarrow
B^{-1+\frac{3}{p}}_{p,\infty}(\R^3)\hookrightarrow BMO^{-1}(\R^3)\hookrightarrow
B^{-1}_{\infty,\infty}(\R^3),
$$
where $3<p<\infty$. The space $BMO^{-1}(\R^3)$ is the largest scaling invariant space known
for which equation (\ref{e1.1}) with $\Omega=0$ is well-posed.

It is a very remarkable fact that the equation (\ref{e1.1}) allows
a global, mild solution for arbitrary large data in the $L^2$-setting provided the
speed $\Omega$ of rotation is fast enough, see \cite{BMN97}, \cite{BMN99} and \cite{CDGG06}.
More precisely, it was proved by Chemin, Desjardins, Gallagher and Grenier in
\cite{CDGG06} that for initial data $u_0 \in L^2(\R^2)^3+ H^{1/2}(\R^3)^3$ satisfying
$\dv u_0 =0$, there exists a constant $\Omega_0>0$ such that for every $\Omega \geq \Omega_0$
the equation \eqref{e1.1} admits a unique, global mild solution. The case of periodic
intial data was considered before by Babin, Mahalov and Nicolaenko in the papers
\cite{BMN97} and \cite{BMN99}.

It is now a natural question to ask whether, for given and fixed $\Omega >0$, there exists
a unique, global mild solution to \eqref{e1.1} provided the initial data is
sufficiently small with respect to the above or related  norms. In this context it is natural to
extend the classical Fujita-Kato approach for the Navier-Stokes equations to the
rotational setting. Hieber and Shibata considered in \cite{HS10} the case of initial data
belonging to $H^{\frac12}(\R^3)$ and proved a global well-posedness result for
\eqref{e1.1} for initial data being small with respect to $H^{\frac{1}{2}}(\R^3)$.
Generalizations of this result to the case of Fourier-Besov spaces are due to
Konieczny and Yoneda \cite{KY11} and Iwabuchi and Takada \cite{IT11}.

More precisely, Konieczy and Yoneda proved the existence of a unique
global mild solution to \eqref{e1.1} for initial data $u_0$ being small with respect to
the norm of $\dot {FB}^{2-\frac{3}{p}}_{p,\infty}(\R^3)$,  where $1<p\leq\infty$.
For the case $p=1$  considered  in \cite{IT11},  the existence of  a unique global mild
solution was  proved provided the initial data $u_0$ are small with respect to
$\dot {FB}^{-1}_{1,2}(\R^3)$. Moreover, it was shown in \cite{IT11} that the space
$\dot {FB}^{-1}_{1,2}(\R^3)$ is critical for the well-posedness of
system (\ref{e1.1}). In fact, it was shown in \cite{IT11} that
equation \eqref{e1.1} is ill-posed in $\dot {FB}^{-1}_{1,q}(\R^3)$ whenever
 $2<q\leq\infty$ and  $\Omega\in\R$.

Giga, Inui, Mahalov and Saal considered in \cite{GIMS07} the problem
of non-decaying initial data and obtained the uniform global solvability of (\ref{e1.1}) in
the scaling invariant space $FM_0^{-1}(\R^3)$. For details, see  \cite{GIMS07} and \cite{GIMS08}.
Note that all of these results rely on good mapping properties  of the Stokes-Coriolis
semigroups on these function spaces.

It seems to be unknown, whether global existence results
are also true for initial data $u_0$ being small with respect to $L^p(\R^3)$ for $p\geq 3$.
The main difficulty here is that Mikhlin's theorem applied to the Stokes-Coriolis semigroup
$T$ yields an estimate of the form
$$
\|T(t)f\|_{L^q}\leq M_p\Omega^2t^2\|f\|_{L^p},\ \ t\geq1, \ \ f\in L^p_\sigma(\R^3),
$$
which is not suitable for fixed point arguments. For this and the definition of $T$
we refer to Section 2 and \cite{HS10}.  Nevertheless, a global existence result for equation
\eqref{e1.1} was recently proved by Chen, Miao and Zhang in \cite{CMZ10} for
highly oscillating initial data in certain hybrid Besov spaces.

The aim of this paper is twofold: first we prove the existence of a unique, global
mild solution to the above problem for initial data $u_0$ being {\em small} in the space
$\dot {FB}^{2-\frac{3}{p}}_{p,r}(\R^3)$, where $1<p\leq\infty$ and $1\leq r\leq\infty$, hereby
generalizing the result in \cite{KY11} for $r=\infty$ to the case $1\leq r\leq\infty$.
We note that Iwabuchi and Takada \cite{IT11} recently proved the well-posedness of
\eqref{e1.1} for data being small with respect to the norm of
$\dot{FB}^{-1}_{1,2}(\R^3)$.

Secondly, considering the two-dimensional situation in the $L^p$-setting, we prove that
\eqref{e1.1} admits a unique, global mild solution $u \in C([0,\infty);L^p(\R^2))$ for
arbitrary $u_0 \in L^p_\sigma(\R^2)$
provided $2\leq q <\infty$. Our argument is based
on applying the curl operator to equation \eqref{e1.1}. The resulting {\em vorticity equation}
allows then for a global estimate in two dimensions which can used to control the term
$\nabla u $ in the $L^p$-norm.

In order to formulate our first result, let us recall the definition of
Fourier-Besov spaces.  To this end, let $\varphi$ be a $C^\infty$ function satisfying
$\hbox{ supp }\varphi\subset\{3/4\leq|\xi|\leq 8/3\}$ and
$$
\sum_{k\in \mathbb{Z}}\varphi(2^{-k}\xi)=1, \quad \xi \in \R^n\backslash\{0\}.
$$
For $k\in \Z$, set $\varphi_k(\xi)=\varphi(2^{-k}\xi)$ and
$h_k=\mathscr{F}^{-1}\varphi_k$. For $s\in\R$ and $1\leq p,r\leq\infty$, the space
$\dot {FB}^s_{p,r}(\R^3)$ is defined to be the set of all $f\in\cS'(\R^3)$
such that $\hat{f}\in L^1_{loc}(\R^3)$  and
$$
\|f\|_{\dot{FB}^s_{p,r}}:=\Big\|\{2^{js}\|\varphi_j\hat{f}\|_{L^p(\R^3)}\}_{j\in\mathbb{Z}}
\Big\|_{l^r} < \infty.
$$
Given $1\leq q \leq \infty$ and $T \in (0,\infty]$, we also make use of
Chemin-Lerner type spaces
$\tilde{L}^q([0,T);\dot {FB}^{s}_{p,r}(\R^3))$, which are defined to be the completion
of $C([0,T];\cS(\R^3))$ with respect to the norm
$$
\|f\|_{\widetilde{L}^q([0,T);\dot {FB}^s_{p,r}(\R^3))}:=
\Big\| \{2^{js}\|\varphi_j\hat{f}\|_{L^q([0,T];L^p(\R^3))}\}_{j\in\mathbb{Z}}\Big\|_{l^r}.
$$
We are now in the position to state our first result.

\vn
\begin{theorem}\label{t1.1}
Let $\Omega\in\R$ and  $1< p\leq\infty$, $1\leq r\leq\infty$. Then there exist
constants $C>0$ and $\ve>0$, independent of $\Omega$, such that for every
$u_0\in \dot {FB}^{2-\frac{3}{p}}_{p,r}(\R^3)$ satisfying $\textrm{div } u_0=0$
and $\|u_0\|_{\dot {FB}^{2-\frac{3}{p}}_{p,r}}\leq \ve$, the equation (\ref{e1.1})
admits a unique, global mild solution $u \in X$, where $X$ is given by
$$
X=\{u\in C([0,\infty);\dot {FB}^{2-\frac{3}{p}}_{p,r}(\R^3)): \|u\|_X\leq C\ve,
   \dv u=0\}
$$
with
$$
\|u\|_X=\|u\|_{\widetilde{L}^\infty([0,\infty);\dot {FB}^{2-\frac{3}{p}}_{p,r}(\R^3))}+
\|u\|_{ \widetilde{L}^1([0,\infty);\dot {FB}^{4-\frac{3}{p}}_{p,r}(\R^3))}.
$$
\end{theorem}

\vn
\begin{remarks}{\rm
a) Observe that due to the results in \cite{IT11}, the above system (\ref{e1.1}) is
ill-posed provided $p=1$ and $r>2$.

\noindent
b) Note that the case $r=\infty$ coincides with the  result of Konieczny and Yoneda
in \cite{KY11}.

\noindent
c) Iwabuchi and Takada \cite{IT11} recently proved the existence of a unique, global mild
solution to equation \eqref{e1.1} for initial data small with respect to the norm of
$\dot {FB}_{1,2}^{-1}$.

\noindent
d) Note that neither $\dot {FB}_{1,2}^{-1}(\R^3) \subset \dot {FB}_{p,r}^{2-3/p}(\R^3)$ for $r\in [1,\infty]$ nor
$\dot {FB}_{p,r}^{2-3/p}(\R^3) \subset \dot {FB}_{1,2}^{-1}(\R^3)$ for $r>2$.
}
\end{remarks}

\vn
Our second result concerning {\em non-small data} in the $L^p(\R^2)$-setting reads as follows.
We denote by $L_\sigma^p(\R^2)$ the solenoidal subspace of $L^p(\R^2)$.

\vn
\begin{theorem}\label{t1.3}
Let $2\leq p<\infty$ and $u_0\in L_\sigma^p(\R^2)$.
Then equation \eqref{e1.1} admits  a
unique, global mild solution $u\in C([0,\infty),L_\sigma^p(\R^2))$.
\end{theorem}

\vn
\section{Linear and Bilinear Estimates}
We start this section by considering the linear Stokes problem with Coriolis force
\begin{equation}\label{e2.1}
\left\{
 \begin{array}{rlll}
 \partial_tu -\Delta u+\Omega e_3\times u+\nabla \pi&=&0,
\ \ &\mathrm{in}\ \R^3\times (0,\infty),\\
  \dv u&=&0, \ \ &\mathrm{in}\ \R^3\times (0,\infty),\\
  u(0)&=&u_0,  \ \ &\mathrm{in}\ \R^3.
   \end{array}
  \right.
\end{equation}
It was shown in \cite{HS10} that the solution of \eqref{e2.1} is given by the Stokes-Coriolis
semigroup $T$, which has the  explicit representation
\begin{equation}\label{semigroup}
T(t)f :=\cF^{-1}\Big[\cos\Big(\Omega\frac{\xi_3}{|\xi|}t\Big)e^{-|\xi|^2t}\mathrm{Id}
\hat{f}(\xi)+\sin\Big(\Omega\frac{\xi_3}{|\xi|}t\Big)e^{-\xi|^2t}R(\xi)\hat{f}(\xi)\Big], \quad
t>0,
\end{equation}
for divergence free vector fields $f \in \cS(\R^3)$.
Here Id is the identity matrix in $\R^3$ and $R(\xi)$ is the skew symmetric matrix defined by
\begin{gather*}R(\xi):=\frac{1}{|\xi|}\left(
\begin{matrix}
0 & \xi_3 &  -\xi_2 \\
-\xi_3 &  0  & \xi_1 \\
\xi_2 &  -\xi_1& \ 0
\end{matrix}\right), \quad \xi \in \R^3\backslash\{0\}.
\end{gather*}
In order to solve equation (\ref{e1.1}), consider the integral equation
$$
\Phi(u):=T(t)u_0-\int_0^t T(t-\tau)\mathbb{P}\mathrm{div}(u\otimes u)(\tau)d\tau,
$$
where $\mathbb{P}:=(\delta_{ij}+R_iR_j)_{1\leq i,j\leq3}$ denotes the Helmholtz projection from
$L^p(\R^3)$ onto its divergence free vector fields. Here $R_i$ denotes the
Riesz transforms for $i=1,2,3$. Since the Riesz transforms $R_i$ are bounded operators
on $\dot{FB}^s_{p,q}$ for all values of $p,q \in [1,\infty]$ and $s \in \R$,
we see that $\mathbb{P}$ defines a bounded operators also on these spaces.

Our first estimate concerns the above convolution integral.

\vn
\begin{lemma}\label{l2.1}
Let $1\leq p,q,a, r\leq\infty$, $s \in \R$ and
$f\in L^a([0,\infty);\dot {FB}^{s}_{p,r}(\R^3))$.
Then there exists a constant $C>0$ such that
$$
\Big\|\int_0^t T(t-\tau)f(\tau)d\tau\Big\|_{\widetilde{L}^q([0,\infty);
\dot {FB}^{s}_{p,r}(\R^3))}\leq C \|f\|_{\widetilde{L}^a([0,\infty);\dot
{FB}^{s-2-\frac{2}{q}+\frac{2}{a}}_{p,r}(\R^3))}.
$$
\end{lemma}

\begin{proof}
By the definition of the norm of $\widetilde{L}^q([0,\infty);\dot {FB}^{s}_{p,r}(\R^3))$,
and by Young's inequality
\begin{eqnarray*}
\Big\|\int_0^t T(t-\tau)f(\tau)d\tau\Big\|_{\widetilde{L}^q([0,\infty);\dot {FB}^{s}_{p,r}
(\R^3))}&=\Big(\sum\limits_k2^{ksr}\Big(\int_0^\infty\Big\|\int_0^t\cF(T(t-\tau)f)
(\tau)d\tau\cdot\varphi_k\Big\|^q_{L^p}dt\Big)^{\frac{r}{q}}\Big)^{\frac{1}{r}}\\
&\leq\Big(\sum\limits_k2^{ksr}\Big(\int_0^\infty\Big(\int_0^te^{-(t-\tau)
\cdot2^{2k}}\|\hat{f}(\tau)
\cdot\varphi_k\|_{L^p}d\tau\Big)^qdt\Big)^{\frac{r}{q}}\Big)^{\frac{1}{r}}\\
&\leq\Big(\sum\limits_k2^{ksr}\Big(\int_0^\infty e^{-t\cdot2^{2k}
\tilde{q}}dt\Big)^{\frac{r}{\tilde{q}}} \Big(\int_0^\infty\|\hat{f}(\tau)
\cdot\varphi_k\|^a_{L^p}\Big)dt\Big)^{\frac{r}{a}}\Big)^{\frac{1}{r}},
\end{eqnarray*}
where $\tilde{q}$ satisfies $1+\frac{1}{q}=\frac{1}{\tilde{q}}+\frac{1}{a}$.
We hence obtain
\begin{eqnarray*}
\Big\|\int_0^t T(t-\tau)f(\tau)d\tau\Big\|_{\widetilde{L}^q([0,\infty);
\dot {FB}^{s}_{p,r}(\R^3))}
&\leq&
C\Big(\sum\limits_k2^{ksr}2^{-2k(1+\frac{1}{q}-\frac{1}{a})r}\|\hat{f}
\cdot\varphi_k\|^r_{L^a([0,\infty);L^p)}\Big)^{\frac{1}{r}}\\
&\leq& C \|f\|_{\widetilde{L}^a([0,\infty);\dot {FB}^{s-2-\frac{2}{q}+\frac{2}{a}}_{p,r}(\R^3))}.
\end{eqnarray*}
\end{proof}

\vn
\begin{lemma}\label{l2.2}
Let $1< p\leq\infty$ and $1\leq p,q,a, r\leq\infty$ and assume that $-1<s<3-\frac{3}{p}$.
Set
$$
Y:=\widetilde{L}^\infty([0,\infty);\dot {FB}^{s}_{p,r}(\R^3))\cap
\widetilde{L}^1([0,\infty);\dot {FB}^{4-\frac{3}{p}}_{p,r}(\R^3)).
$$
Then there exists a constant $C>0$ such that
$$
\|uv\|_{\widetilde{L}^1([0,\infty);\dot {FB}^{s+1}_{p,r}(\R^3))}\leq C \|u\|_Y\|v\|_Y.
$$
\end{lemma}

\vn
\begin{proof}
Let $\varphi$ and $h_k$ be defined as in Section 1 for $k \in \Z$.
Define the homogeneous dyadic blocks $\dot{\Delta}_k$ by
$$
\dot{\Delta}_{k}u: =\varphi(2^{-k}D)u=\int_{\R^N}h_k(y)u(x-y)dy, \quad k\in \Z,
$$
and for $j \in \Z$, set  $\dot{S}_ju: =\sum\limits_{k=-\infty}^j\dot{\Delta}_ku$.
We then obtain

\begin{equation}\label{e2.2}
\|uv\|_{\widetilde{L}^1([0,\infty);\dot {FB}^{s+1}_{p,r}(\R^3))}=
\Big(\sum\limits_j2^{j(s+1)r}\Big(\int_0^\infty\|\widehat{\dot{\Delta}_j(uv)}
\|_{L^p}dt\Big)^r\Big)^{\frac{1}{r}}.
\end{equation}
Using the Bony decomposition \cite{Bon81}, \cite{CM97} and \cite{BCD11},  we
rewrite $\dot{\Delta}_j(uv)$ as
\begin{equation}\label{e2.3}
\dot{\Delta}_j(uv)=\sum\limits_{|k-j|\leq4}\dot{\Delta}_j(\dot{S}_{k+1}u\dot{\Delta}_kv)
+\sum\limits_{|k-j|\leq4}\dot{\Delta}_j(\dot{S}_{k+1}v\dot{\Delta}_ku)
+\sum\limits_{k\geq j-2}\dot{\Delta}_j(\dot{\Delta}_{k}u\tilde{\dot{\Delta}}_kv) =:
I + II + III.
\end{equation}
Then, by triangle inequalities in $L^p(\R^3)$ and $l^r(\mathbb{Z})$, we have
$$\aligned
\|uv\|_{\widetilde{L}^1([0,\infty);\dot {FB}^{s+1}_{p,r}(\R^3))}&\leq \Big(\sum\limits_j2^{j(s+1)r}\Big(\int_0^\infty\|\widehat{\dot{\Delta}_jI}
\|_{L^p}dt\Big)^r\Big)^{\frac{1}{r}}+\Big(\sum\limits_j2^{j(s+1)r}\Big(\int_0^\infty\|\widehat{\dot{\Delta}_jII}
\|_{L^p}dt\Big)^r\Big)^{\frac{1}{r}}\\
&\ \ \ +\Big(\sum\limits_j2^{j(s+1)r}\Big(\int_0^\infty\|\widehat{\dot{\Delta}_jIII}
\|_{L^p}dt\Big)^r\Big)^{\frac{1}{r}}\\
=:J_1+J_2+J_3.
\endaligned
$$
For the term $J_1$, we note
$$
J_1=\Big(\sum\limits_j2^{j(s+1)r}\Big(\int_0^\infty\|\sum\limits_{|k-j|\leq4}
\widehat{\dot{\Delta}_j
(\dot{S}_{k+1}u\dot{\Delta}_kv)}\|_{L^p}dt\Big)^r\Big)^{\frac{1}{r}}.
$$
For fixed $j$, Lemma \ref{l2.1} yields
$$
\aligned
2^{j(s+1)}&\int_0^\infty\|\sum\limits_{|k-j|\leq4}\widehat{\dot{\Delta}_j
(\dot{S}_{k+1}u\dot{\Delta}_kv)}\|_{L^p}dt\\
&\leq 2^{j(s+1)}\int_0^\infty\sum\limits_{|k-j|\leq4}\|
\varphi_j(\chi_k\hat{u}\ast\varphi_k\hat{v})\|_{L^p}dt\\
&\leq 2^{j(s+1)}\int_0^\infty\sum\limits_{|k-j|\leq4}\|\chi_k\hat{u}\|_{L^1}
\|\varphi_k\hat{v})\|_{L^p}dt\\
&\leq 2^{j(s+1)}\int_0^\infty\sum\limits_{|k-j|\leq4}
\sum\limits_{k'\leq k}\|\varphi_{k'}\hat{u}\|_{L^p}2^{k'(3-\frac{3}{p})}
\|\varphi_k\hat{v})\|_{L^p}dt\\
&\leq 2^{j(s+1)}\int_0^\infty\sum\limits_{|k-j|\leq4}\Big(\sum\limits_{k'\leq k}
\|\varphi_{k'}\hat{u}\|^r_{L^p}2^{k'sr}\Big)^{\frac{1}{r}}
\Big(\sum\limits_{k'\leq k}2^{k'r'(3-\frac{3}{p}-s)}\Big)^{\frac{1}{r'}}\|
\varphi_k\hat{v})\|_{L^p}dt\\
&\leq C\|u\|_{\widetilde{L}^\infty([0,\infty);
\dot {FB}^{s}_{p,r})}2^{j(s+1)}\int_0^\infty\sum\limits_{|k-j|\leq4}2^{k(3-\frac{3}{p}-s)}\|\varphi_k\hat{v})\|_{L^p}dt.
\endaligned$$
Hence, by Young's inequality,
$$
\aligned
J_1&\leq C\|u\|_{\widetilde{L}^\infty([0,\infty);
\dot {FB}^{s}_{p,r})}\Big(\sum\limits_j\Big(\sum\limits_{|k-j|\leq4}2^{k(4-\frac{3}{p})}
2^{(j-k)(s+1)}\|\varphi_k\hat{v})
\|_{L^1([0,\infty);L^p)}\Big)^r\Big)^{\frac{1}{r}}\\
&\leq C \|u\|_{\widetilde{L}^\infty([0,\infty);
\dot {FB}^{s}_{p,r})}\|v\|_{\widetilde{L}^1([0,\infty);\dot {FB}^{4-\frac{3}{p}}_{p,r})}.
\endaligned$$
The term  $J_2$ is estimated in the same way as $J_1$. In fact,
$$
\aligned
J_2&\leq C\|v\|_{\widetilde{L}^\infty([0,\infty);
\dot {FB}^{s}_{p,r})}\Big(\sum\limits_j\Big(\sum\limits_{|k-j|\leq4}2^{k(4-\frac{3}{p})}2^{(j-k)(s+1)}
\|\varphi_k\hat{u})
\|_{L^1([0,\infty);L^p)}\Big)^r\Big)^{\frac{1}{r}}\\
&\leq C \|v\|_{\widetilde{L}^\infty([0,\infty);
\dot {FB}^{s}_{p,r})}\|u\|_{\widetilde{L}^1([0,\infty);\dot {FB}^{4-\frac{3}{p}}_{p,r})}.
\endaligned
$$
Finally,  we focus on the third term $J_3$. As in the estimate to $J_1$, for fixed $j$, we obtain
$$
\aligned
2^{j(s+1)}&\int_0^\infty\|\sum\limits_{k\geq j-2}
\sum\limits_{|k-k'|\leq 1}\varphi_j(\varphi_k\hat{u}\ast\varphi_{k'}\hat{v})\|_{L^p}dt\\
&\leq 2^{j(s+1)}\int_0^\infty\sum\limits_{k\geq j-2}\sum
\limits_{|k-k'|\leq 1}\|\varphi_{k'}\hat{v})\|_{L^p}
\|\varphi_k\hat{u}\|_{L^p}2^{k(3-\frac{3}{p})}dt\\
&\leq C2^{j(s+1)}\int_0^\infty\sum\limits_{k\geq j-2}
\Big(\sum\limits_{|k-k'|\leq 1}\|\varphi_{k'}\hat{v})\|^r_{L^p}2^{k'rs}\Big)^{\frac{1}{r}}
2^{-ks}2^{k(3-\frac{3}{p})}\|\varphi_k\hat{u}\|_{L^p}dt\\
&\leq C\|v\|_{\widetilde{L}^\infty([0,\infty);\dot {FB}^{s}_{p,r})}
\sum\limits_{k\geq j-2}2^{(j-k)(s+1)}2^{k(4-\frac{3}{p})}
\int_0^\infty\|\varphi_k\hat{u}\|_{L^p}dt.
\endaligned
$$
Thus,  by Young's inequality
$$
J_3\leq C \|v\|_{\widetilde{L}^\infty([0,\infty);\dot {FB}^{s}_{p,r})}
\|u\|_{\widetilde{L}^1([0,\infty);\dot {FB}^{4-\frac{3}{p}}_{p,r})}
$$
since $s>-1$.

Summing up, we see  that
$$
\|uv\|_{\widetilde{L}^1([0,\infty);\dot {FB}^{s+1}_{p,r}(\R^3))}\leq C \|u\|_Y\|v\|_Y.
$$
\end{proof}

We conclude this section with the following Lemma.

\vn
\begin{lemma}\label{l2.3} \cite{Tri83}
Let $1\leq p,q\leq \infty$, $0<r<R<\infty$, $j\in\Z$ and $n \in \N$.
Then, for any multiindex $\gamma \in \N^n$, the following estimates hold:
\begin{itemize}
\item[a)]  If $\mathrm{supp }\hat{f} \subset\{\xi \in \R^n:|\xi|\leq R2^j\}$, then
$\|(i\cdot)^\gamma\hat{f}\|_{L^q(\R^n)}\leq 2^{j|\gamma|+nj(\frac{1}{q}-\frac{1}{p})}
\|\hat{f}\|_{L^p(\R^n)}$.
\item[b)] If $\mathrm{supp } \hat{f} \subset\{\xi \in \R^n: r2^j\leq|\xi|\leq R2^j\}$,
then
$\|\hat{f}\|_{L^q(\R^n)}\leq 2^{-j|\gamma|}\sup\limits_{|\beta|=|\gamma|}\|(i\cdot)^\beta
\hat{f}\|_{L^q(\R^n)}$.
\end{itemize}
\end{lemma}

\section{Proof of Theorem \ref{t1.1}}

\vn
For the proof of Theorem \ref{t1.1} we make use of the following
standard fixed point result. For a proof, we refer e.g. to \cite{CK04}.

\begin{proposition}\label{t3.1}
Let $X$ be a Banach space and $B:\ X\times X\rightarrow X$ be a bounded
bilinear form satisfying $\|B(x_1,x_2)\|_X\leq \eta \|x_1\|\|x_2\|_X$ for all
$x_1,x_2\in X$ and a constant $\eta>0$. Then, if $0<\eps<\frac{1}{4\eta}$ and if
$a\in X$ such that $\|a\|_X\leq \eps$, the equation $x=a+B(x,x)$ has a solution in $X$
such that $\|x\|_X\leq 2\eps$. This solution is the only one in the
ball $\overline{B}(0,2\eps)$. Moreover, the solution depends continuously on $a$ in the
following sense: if $\|\tilde{a}\|_X\leq\eps, $ $\tilde{x}=\tilde{a}+B(\tilde{x},
\tilde{x})$, and $\|\tilde{x}\|_X\leq2\eps$, then
$$\|x-\tilde{x}\|_X\leq\frac{1}{1-4\eta\eps}\|a-\tilde{a}\|_X.$$
\end{proposition}

\vn
In the following, we choose an underlying Banach space $X$ given by
$$
X:=\widetilde{L}^\infty([0,\infty);\dot {FB}^{2-\frac{3}{p}}_{p,r}(\R^3))
\cap \widetilde{L}^1([0,\infty);\dot{FB}^{4-\frac{3}{p}}_{p,r}(\R^3)),
$$
and recall that $\Phi$ was defined by
$$
\Phi(u)=T(t)u_0 - \int_0^t T(t-\tau)\mathbb{P}\mathrm{div}(u\otimes u)(\tau)d\tau.
$$
We estimate first the term $T(t)u_0$.

\begin{lemma}\label{l3.2}
Let $p,r\in [1,\infty]$, $s=2-3/p$  and  $u_0\in\dot {FB}^{s}_{p,r}(\R^3)$. Then there exists a
constant $C>0$ such that
\begin{eqnarray}
\|T(t)u_0\|_{\widetilde{L}^\infty([0,\infty);\dot {FB}^{s}_{p,r})}
&\leq& C\|u_0\|_{\dot {FB}^{s}_{p,r}}, \quad t>0,\label{e3.1}\\
\label{e3.2}
\|T(t)u_0\|_{\widetilde{L}^1([0,\infty);\dot {FB}^{s+2}_{p,r})}&\leq& C
\|u_0\|_{\dot {FB}^{s}_{p,r}}, \quad t>0.
\end{eqnarray}
\end{lemma}

\begin{proof}
We prove first estimate (\ref{e3.1}). By the definition of the norm, we have
$$
\|T(t)u_0\|_{\widetilde{L}^\infty([0,\infty);\dot{FB}^{2-\frac{3}{p}}_{p,r})} \leq
\Big(\sum\limits_k2^{k(2-\frac{3}{p})r}
\sup\limits_{t\in[0,\infty)}\|\varphi_k\hat{u_0}\|^r_{L^p}\Big)^{\frac{1}{r}}
\leq C \|u_0\|_{\dot {FB}^{2-\frac{3}{p}}_{p,r}}.
$$
In order to prove the second  estimate \eqref{e3.2} above, note that
$$
\|T(t)u_0\|_{\widetilde{L}^1([0,\infty);\dot{FB}^{4-\frac{3}{p}}_{p,r})}
\leq \Big(\sum\limits_k2^{k(4-\frac{3}{p})r}
\Big(\int_0^\infty e^{- t2^{2k}}\|\varphi_k\hat{u_0}\|_{L^p}dt\Big)^r\Big)^{\frac{1}{r}}
\leq C\|u_0\|_{\dot{FB}^{2-\frac{3}{p}}_{p,r}}.
$$

\end{proof}

We next consider the bilinear operator $B$ given  by
$$
B(u,v):=\int_0^t T(t-\tau)\mathbb{P}\mathrm{div}(u\otimes v)d\tau.
$$
By Lemma Lemma \ref{l2.1}, Lemma \ref{l2.3} and Lemma \ref{l2.2} with $s=2-\frac{3}{p}$,
we obtain
$$
\aligned
\|B(u,v)\|_{\widetilde{L}^1([0,\infty);\dot{FB}^{4-\frac{3}{p}}_{p,r}(\R^3))}&
=\Big\|\int_0^t T(t-\tau)\mathbb{P}\mathrm{div}(u\otimes v)
d\tau\Big\|_{\widetilde{L}^1([0,\infty);\dot {FB}^{4-\frac{3}{p}}_{p,r}(\R^3))}\\
&\leq C\|\mathrm{div}(u\otimes v)\|_{\widetilde{L}^1([0,\infty);
\dot {FB}^{2-\frac{3}{p}}_{p,r})}\\
&\leq C\|uv\|_{\widetilde{L}^1([0,\infty);\dot {FB}^{3-\frac{3}{p}}_{p,r})}\\
&\leq C\|u\|_X\|v\|_X.
\endaligned
$$
Similarly,
$$
\aligned
\|B(u,v)\|_{\widetilde{L}^\infty([0,\infty);\dot {FB}^{2-\frac{3}{p}}_{p,r}(\R^3))}&=
\Big\|\int_0^t T(t-\tau)\mathbb{P}\mathrm{div}(u\otimes v)d\tau\Big\|_{\widetilde{L}^\infty
([0,\infty);\dot {FB}^{2-\frac{3}{p}}_{p,r}(\R^3))}\\
&\leq C\|\mathrm{div}(u\otimes v)\|_{\widetilde{L}^1([0,\infty);
\dot {FB}^{2-\frac{3}{p}}_{p,r})}\\
&\leq C\|uv\|_{\widetilde{L}^1([0,\infty);\dot {FB}^{3-\frac{3}{p}}_{p,r})}\\
&\leq C\|u\|_X\|v\|_X.
\endaligned
$$
Thus, combining these estimates with Lemma \ref{l3.2} yields
$$
\|\Phi(u)\|_X\leq C\|u_0\|_{\dot {FB}^{2-\frac{3}{p}}_{p,r}}+ 4C\ve^2,
$$
as well as
$$
\|\Phi(u)-\Phi(v)\|_X\leq C(\|u\|_X+\|v\|_X)\|u-v\|_X.
$$
Choosing now $\ve\leq\frac{1}{8C}$, for every $u_0\in\dot {FB}^{2-\frac{3}{p}}_{p,r}(\R^3)$
with $\|u_0\|_{\dot {FB}^{2-\frac{3}{p}}_{p,r}} \leq
\frac{\ve}{C}$, we finally obtain
$$
\|\Phi(u)\|_X \leq 2\ve \mbox{ and}
$$
$$\|\Phi(u)-\Phi(v)\|_X \leq \frac{1}{2}\|u-v\|_X.
$$
Applying Proposition \ref{t3.1} to the given situation completes the proof of
Theorem \ref{t1.1}.

\vn
\section{Global existence for non-small data in $L^p_\sigma(\R^2)$ }

In this section we consider equation \eqref{e1.1} in the two-dimensional setting
and in the case where the initial data $u_0$ belong to $L^p_\sigma(\R^2)$ for $p>2$.
To this end, we note first that  the equations of
Navier-Stokes with Coriolis force are equivalent to the Navier-Stokes equations with
linearly growing initial data. Indeed, we may rewrite  equation \eqref{e1.1} for a
two dimensional rotating fluid  as
\begin{equation}\label{e4.1}
 \left\{
 \begin{array}{rlll}
 \partial_tu-\Delta u+u\cdot\nabla u-2Mu+\nabla \pi&=&0, \ \ &y\in\R^2,\ t>0,\\
  \hbox{div }u&=&0, \ \ &y\in\R^2,\ t>0,\\
  u(0)&=&u_0,  \ \ &y\in\R^2,
   \end{array}
  \right.
\end{equation}
where $M$ is given by
\begin{gather*}M=-\frac{\Omega}{2}\left(
\begin{matrix}
0 & -1   \\
1 &  0   \\
\end{matrix}\right).
\end{gather*}
Then, by the change of variables $x=e^{-tM}y$ and by setting
$$
v(t,x):=e^{-tM}u(t,e^{tM}x), \qquad q(t,x):=\pi(t,e^{tM}x),
$$
we obtain the following set of equations for $v$
\begin{equation}\label{e4.2}
 \left\{
 \begin{array}{rlll}
 \partial_tv-\Delta v+v\cdot\nabla v-Mx\cdot\nabla v-Mv+\nabla q&=&0, \ \ &x\in\R^2,\ t>0,\\
  \hbox{div }v&=&0, \ \ &x\in\R^2,\ t>0,\\
  v(0)&=&u_0,  \ \ &x\in\R^2.
   \end{array}
  \right.
\end{equation}
These are the usual equations of Navier-Stokes with  linearly growing initial data.
Indeed, setting $U=v-Mx$, we have
\begin{equation}\label{e4.3}
 \left\{
 \begin{array}{rlll}
 \partial_tU-\Delta U+U\cdot\nabla U+\nabla \tilde{\pi}&=&0, \ \ &x\in\R^2,\ t>0,\\
  \hbox{div }U&=&0, \ \ &x\in\R^2,\ t>0,\\
  U(0)&=&u_0-Mx,  \ \ &x\in\R^2,
   \end{array}
  \right.
\end{equation}
with $\nabla \tilde{\pi}=\nabla q -MMx$.
For initial data $u_0\in L_\sigma^p(\R^2)$, it was shown in  \cite{HS05} that there exists
 a unique, local  mild solution $v$ to equation (\ref{e4.2}) in the space
$C([0,T_0);L_\sigma^p(\R^2))$, where $2\leq p<\infty$.
We note that if $u_0\in L_\sigma^p(\R^2)$, Theorem 2.1 in  \cite{HS05} implies that
$t^{\frac{1}{2}}\nabla v\in C([0,T_0);L^p(\R^2))$. Thus there exists $t_1\in (0,T_0)$
such that $\nabla v(t_1)\in L^p(\R^2)$ which implies that
$\mathrm{rot}\ v(t_1)\in L^p(\R^2)$. Hence, in order to prove
Theorem \ref{t1.3}, it suffices to show an a priori estimate of the following form. In the sequel, we set $w:=\mathrm{rot}\ v$.

\vn
\begin{proposition}\label{t4.1}
Let  $2\leq p<\infty$ and $v(t_1)\in L_\sigma^p(\R^2)$ such that $\textrm{rot }v(t_1) \in L^p(\R^2)$ for some $t_1\in (0,T_0)$.
Let $v$ be the mild solution of (\ref{e4.2}). Then there exists a constant $C>0$ such that
$$
\|v(t)\|_{L^p}\leq C \|v(t_1)\|_{L^p}\exp\Big(Ct\|w(t_1)\|_{L^p}\Big),\quad t>t_1,
$$
where $w(t_1)=\mathrm{rot}\ v(t_1)$.
\end{proposition}

\begin{proof}
Consider the operator $A$ in $L^p_\sigma(\R^2)$ given by
$$
Au:= -\Delta u - <M\cdot,\nabla u> + Mu
$$
equipped with the domain $D(A)=\{u\in W^{2,p}(\R^2):<M\cdot,\nabla u> \in L^p(\R^2)\}$.
By the results in \cite{HS05}, the mild solution of \eqref{e4.2} is represented by
$$
v(t)=e^{-tA}v(t_1)-\int_{t_1}^te^{-(t-s)A}\mathbb{P}(v\cdot\nabla v)(s)ds+
2\int_{t_1}^te^{-(t-s)A}\mathbb{P}(Mv)(s)ds,
$$
for $t>t_1$. Applying Proposition 3.4 in \cite{HS05} yields
\begin{equation}\label{e4.4}
\|e^{-(t-s)A}\mathbb{P}(v\cdot\nabla v)(s)\|_{L^p}
\leq  \frac{C}{(t-s)^{\frac{1}{p}}}\cdot\|v\cdot\nabla v(s)\|_{L^{\frac{p}{2}}}
\leq  \frac{C}{(t-s)^{\frac{1}{p}}}\cdot\|v(s)\|_{L^q}\cdot\|\nabla v(s)\|_{L^p},
\quad t>s>t_1.
\end{equation}
Employing the inequality $\|\nabla v(s)\|_{L^p}\leq \frac{p^2}{p-1}\|w(s)\|_{L^p}$
(see e.g. \cite{BCD11}), we see that
$$
\|v(t)\|_{L^p}\leq C\|v(t_1)\|_{L^p}+ C \int_{t_1}^t\frac{C}{(t-s)^{\frac{1}{q}}}
\cdot\|v(s)\|_{L^p}\cdot\|w(s)\|_{L^p}ds+C\int_{t_1}^t\|v(s)\|_{L^p}ds, \quad t>t_1.
$$
Next, applying curl  to equation \eqref{e4.2}, we verify that the vorticity $w= \mbox{rot }v$
satisfies the equation
\begin{equation}\label{e4.5}
\left\{
 \begin{array}{rlll}
 \partial_tw-\Delta w+v\cdot\nabla w-Mx\cdot\nabla w&=&0, \ \ &x\in\R^2,\ t>0,\\
 w(0)&=&\hbox{rot}\ u_0. &
   \end{array}
  \right.
\end{equation}
A standard energy estimate allows us to show that
$$
\|w(t)\|_{L^p}\leq C \|w(t_1)\|_{L^p}, \quad t>t_1.
$$
Hence, we have
$$
\|v(t)\|_{L^p}\leq C\|v(t_1)\|_{L^p}+C\|w(t_1)\|_{L^p}\int_{t_1}^t
\Big(\frac{1}{(t-s)^{\frac{1}{p}}}+1\Big)\|v(s)\|_{L^p}ds, \quad t>t_1.
$$
Finally, Gronwall's inequality yields the  desired estimate. This finishes the proof of
Proposition \ref{t4.1}.
\end{proof}

\vn
By Proposition \ref{t4.1}, we obtain a unique, global solution $\widetilde{v}$ of \eqref{e4.2}
on $[t_1,\infty)$ for the  initial data $v(t_1)$. A uniqueness argument ensures that
$v(t)=\widetilde{v}(t)$ on $[t_1,T_0)$. Therefore,
the local solution $v$ on $[0,T_0)$ can be continued globally. This finishes the proof of
Theorem \ref{t1.3}.

\vn 
{\bf Acknowledgments} The main part of this work was carried
out while the first and the second authors are visiting the
Department of Mathematics at the  Technical University of Darmstadt. They
would express their gratitude to Prof. Matthias Hieber for his kind
hospitality and the Deutsche Forschungsgemeinschaft (DFG) for financial support.


\begin{thebibliography}{999}


\bibitem{BCD11} H. Bahouri, J. Y.  Chemin and R. Danchin,
{\em Fourier Analysis and Nonlinear Partial Differential Equations},
Springer Grundlehren, 343, (2011).

\bibitem{BMN97} A. Babin, A. Mahalov, and B. Nicolaenko,
{Regularity and integrability of 3D Euler and Navier-Stokes equations for rotating fluids},
{\em Asymptot. Anal.} 15, (1997), 103-150.

\bibitem{BMN99} A. Babin, A. Mahalov, and B. Nicolaenko,
{Global regularity of 3D rotating Navier-Stokes equations for resonant domains},
{\em Indiana Univ. Math. J.} 48, (1999), 1133-1176.


\bibitem{Bon81} J. M. Bony,
{Calcul symbolique et propagation des singularit$\acute{e}$s pour $\acute{e}$quations
aux d$\acute{e}$riv$\acute{e}$es partielles nonlin$\acute{e}$aires}.
{\em Ann. Sci. L'$\acute{E}$cole Normale
Sup$\acute{e}$rieure} 14 (1981), 209-246.

\bibitem{CM97} M. Cannone and Y. Meyer
{Littlewood-Paley decompositions  and Navier-Stokes Equations},
{\em Methods and Application in Analysis} 2, (1997), 307-319.


\bibitem{CDGG06} J. Y. Chemin, B. Desjardins, I. Gallagher, and E. Grenier,
{\em Mathematical Geophysics}, Oxford Lecture Series in Mathematics and its Applications,
Oxford University Press, Oxford, 2006.


\bibitem{CK04} M. Cannone, G.Karch,
{Smooth or singular solutions to the Navier-Stokes system}, {\em J. Diff. Equ.}  197, (2004),
 247-274.

\bibitem{CMZ10}Q. Chen, C. Miao, Z. Zhang,
{Global well-posedness for the 3D rotating Navier-Stokes equations with highly oscillating
initial data}, Arxiv: 0910.3064v2(2010).

\bibitem{FWZ11} D. Fang, S. Wang, T. Zhang,
{Wellposedness for anisotropic rotating fuid equations}, {\em
Appl. Math. J. Chinese Univ.}, to appear.


\bibitem{GIMM05} Y. Giga, K. Inui, A. Mahalov, S. Matsui,
{Uniform local solvability for the Navier-Stokes equations with the Coriolis force},
{\em Methods and Applications of Analysis} 12, (2005), 381-394.

\bibitem{GIMS07} Y. Giga, K. Inui, A. Mahalov, J. Saal,
{Global solvabiliy of the Navier-Stokes equations in spaces based on sum-closed frequency
sets}, {\em Adv. Diff. Equ.} 12, (2007), 721-736.

\bibitem{GIMS08} Y. Giga, K. Inui, A. Mahalov, J. Saal,
{Uniform global solvability of the rotating Navier-Stokes equations for nondecaying initial
data}, {\em Indiana Univ. Math. J.} 57, (2008), 2775-2791.



\bibitem{HS05} M. Hieber, O. Sawada,
{The Navier-Stokes equations in $\R^N$ with linearly growing initial data},
{\em Arch. Ratinoal Mech. Anal}. 175, (2005), 269-285.

\bibitem{HS10} M. Hieber, Y. Shibata,
{The Fujita-Kato approach to the  Navier-Stokes equations in the rotational framework},
{\em  Math. Z.} 265, (2010), 481-491.


\bibitem{Kat84} T. Kato, \textit{Strong $L^p$ solutions of Navier-Stokes equations in
$\R^N$ with applications to weak solutions}, Math. Z. 187, 471-480(1984).

\bibitem{KY11} P. Konieczny, T. Yoneda, {On dispersive effect of the Coriolis force for the
stationary Navier-Stokes equations}, {\em J. Diff. Equ.} 250, (2011), 3859-3873.


\bibitem{Maj03} A. Majda,
{\em Introduction to PDEs and Waves for the Atmosphere and Ocean}.
Courant Lecture Notes in Math., 2003.



\bibitem{IT11} T. Iwabuchi, R. Takada,
{Global well-posedness for the Navier-Stokes equations with the Coriolis force in
function spaces of Besov type}, Preprint 2011.


\bibitem{Tri83} H. Triebel,
{\em Theory of Function Spaces},
Monographs in Mathematics, 78. Birkh\"auser Verlag, Basel, 1983.

\end{thebibliography}
\end{document}